\documentclass[english]{amsart}
\usepackage[T1]{fontenc}
\usepackage[latin9]{inputenc}
\usepackage{amsmath}
\usepackage{amssymb}
\usepackage{amsmath}
\usepackage{amstext}
\usepackage{amsfonts}
\usepackage{amscd}
\usepackage{amsbsy}
\usepackage{amsthm}
\usepackage{ifthen}
\usepackage{enumerate}
\usepackage{latexsym}
\usepackage{babel}
\usepackage{xy}
\usepackage{graphics}
\xyoption{all}

\makeatletter

\makeatletter
\def\input@path{{/home/filiuspelei/Hope//}}
\makeatother

\newcounter{dummy} \numberwithin{dummy}{section}

\newtheorem{theorem}[dummy]{Theorem}
\newtheorem{prop}[dummy]{Proposition}
\newtheorem{lemma}[dummy]{Lemma}

\newtheorem{corollary}[dummy]{Corollary}
\newtheorem{conjecture}[dummy]{Conjecture}
\newtheorem{question}{Question}
\theoremstyle{definition}
\newtheorem{definition}[dummy]{Definition}
\newtheorem{remark}[dummy]{Remark}
\newtheorem{example}[dummy]{Example}

\DeclareMathOperator{\tr}{Tr}

\DeclareMathOperator{\h}{ht}
\DeclareMathOperator{\coker}{coker}

\DeclareMathOperator{\spec}{spec}

\DeclareMathOperator{\hm}{Hom}
\DeclareMathOperator{\ext}{Ext}
\DeclareMathOperator{\tor}{Tor}

\DeclareMathOperator{\md}{mod}

\DeclareMathOperator{\grade}{grade}
\DeclareMathOperator{\zz}{\Omega}
\DeclareMathOperator{\pd}{pd}

\DeclareMathOperator{\depth}{depth}

\DeclareMathOperator{\ann}{ann}

\DeclareMathOperator{\thick}{Thick}

\DeclareMathOperator{\supp}{supp}
\DeclareMathOperator{\cx}{cx}
\DeclareMathOperator{\codim}{codim}
\DeclareMathOperator{\ass}{Ass}
\DeclareMathOperator{\ABdim}{AB-dim}
\DeclareMathOperator{\join}{Join}

\DeclareMathOperator{\sing}{Sing}
\DeclareMathOperator{\proj}{Proj}
\DeclareMathOperator{\maxspec}{maxspec}
\DeclareMathOperator{\lin}{line}
\DeclareMathOperator{\sng}{Sing}

\newcommand{\NN}{\mathbb{N}}

\newcommand{\QQ}{\mathbb{Q}}

\newcommand{\AAA}{\mathbb{A}}
\newcommand{\PPP}{\mathbb{P}}

\newcommand{\C}{\mathcal{C}}

\newcommand{\M}{\mathcal{M}}

\newcommand{\Ss}{\mathfrak{S}}

\newcommand{\mm}{\mathfrak{m}}

\newcommand{\al}{\alpha}
\newcommand{\be}{\beta}

\newcommand{\ld}{\lambda}

\newcommand{\ph}{\varphi}

\newcommand{\sbe}{\subseteq}

\newcommand{\tns}{\otimes}

\newcommand{\xto}{\xrightarrow}

\newcommand{\del}{\backslash}

\newcommand{\dell}{\partial}

\newcommand{\bl}{\bullet}

\newcommand{\MCM}{\mbox{MCM}}

\begin{document}

\title{Cohomological support and the geometric join}

\author{Hailong Dao}
\address{Department of Mathematics\\
University of Kansas\\
Lawrence, KS 66045-7523, USA}
\email{hdao@ku.edu}

\author{William Sanders }
\address{Department of Mathematical Sciences\\
Norwegian University of Science and Technology\\ 
NO-7491 Trondheim, Norway}
\email{william.sanders@math.ntnu.no}

\begin{abstract}
Let $M,N$ be finitely generated modules over a local complete intersection $R$. Assume that for each $i>0$, $\tor^R_i(M,N)=0$. We prove that the cohomological support  of $M\otimes_R N$ (in the sense of Avramov-Buchweitz) is equal to the geometric join of the cohomological supports of $M,N$. Such result gives a new connection between two active areas or research, and immediately produces several surprising corollaries.  Naturally, it also  raises many intriguing new questions about the homological properties of modules over a complete intersection, some of those are investigated in the second half of this note.   

\end{abstract}

\thanks{This work would not have been possible without the wonderful hospitality and stimulating environment of various institutions:  MSRI, CRM, Dalhousie University and the University of Kansas. We would also like to thank David Eisenbud and Srikanth Iyengar for helpful comments and encouragements throughout this project.}

\maketitle

\section{Introduction}
Let $(R, \mm,k)$ be a local complete intersection and $M$ a finitely generated $R$-module. Inspired by the ideas of Quillen in modular representations context, Avramov and Buchweitz in \cite{AvramovBuchweitz00} defined a geometric object attached to the total Ext module $\bigoplus \ext^i_R(M,k)$. It was originally called the support variety, or cohomological support  of $M$, and denoted by $V^*(M)$ (see Section 2 for details). In this paper, we shall refer to this object  as the cohomological support of $M$. It is a closed subscheme of $ \PPP^{c-1}_k$, where $c$ is the codimension of $R$.  

The following is an immediate consequence of the theory  of cohomological supports developed in \cite{AvramovBuchweitz00}: if $M,N$ are Tor-independent, i.e. $\tor_i^R(M,N)=0$ for $i>0$, then $V^*(M), V^*(N)$ are disjoint and
\[\dim V^*(M\otimes_R N)= \dim V^*(M) + \dim V^*(N) +1.\]

One of our main results in this paper  gives a geometric clarification of this formula, by proving that under the above hypothesis, the  cohomological support of $M\otimes_R N$ is actually the join of $V^*(M), V^*(N)$. Here, the join of two closed subschemes is the closure of the union of all the lines joining two points, one from each subscheme. 
\begin{theorem}(Theorem \ref{supvarmain1})
Let $R$ be a local complete intersection, and nonzero $M,N\in\md(R)$.  If $\tor_{>0}(M,N)=0$, then $V^*(M\tns N)=\join(V^*(M),V^*(N))$.  

\end{theorem}

The condition that $\tor_{>0}(M,N)=0$ may appear restrictive, but it is quite natural in this context. Namely, for each pair of disjoint subschemes $U,V$, one can attach modules, in fact, whole categories of modules $X,Y$ such that $V^*(M)= U$ and $V^*(N)=V$ for each $M\in X, N\in Y$. Then the vanishing of all higher Tor modules is automatic as long as $\depth M+ \depth N$ is big enough, which one can guarantee by taking syzygies.

Thus, our Theorem can be viewed as a ``categorification" of the join operation. At any rate, it provides a link between the theory of cohomological support to a very classical concept of algebraic geometry. Unsurprisingly, this immediately leads to many corollaries and interesting  questions, some of them we also address in this work. 

The proof of the above result  occupies Section 3. It combines homological and geometric techniques (the preparatory materials are collected in Section 2).  In Section 4 we give some quick corollaries, including a dual version for $\ext$ (Theorem \ref{main4}). Section 5 is devoted to a study of the Tor-independent condition between two modules $M$ and $N$. Our main result here is a sufficient condition for two modules to be Tor-independent, Theorem \ref{supvarmain3}. The novelty here is that the condition involves dimension instead of depth. Motivated by this result, in Section 6 we discuss some  relationships between the Tor-independent property and the dimension of the modules. These are inspired by the classical ``homological conjectures" in commutative algebra, and certain statements we study can be viewed as  generalizations of them. Finally, Section 7 collects examples (sometimes with the help of Macaulay 2) and open problems suggested by our results.

\section{Background}

\subsection{The ring of cohomological operators}

The study of cohomological supports over complete intersection rings was initiated by Avramov and Buchweitz in \cite{AvramovBuchweitz00}.  For the entirety of this section,  $(R,\mm,k)$ will be a local complete intersection of codimension $c$ such that $\hat{R}=Q/(f_1,\dots,f_c)$ where $Q$ is a regular local ring and $\underline{f}=f_1,\dots,f_c$ a regular sequence not contained in the square of the maximal ideal of $Q$.  Let $\tilde{k}$ be the algebraic closure of $k$.   The cohomological support of a finitely generated $R$-module  $M$ is essentially the support of $\ext_{\hat{R}}(\hat{M},k)$ as a module over the polynomial ring $S=k[\chi_1,\dots,\chi_c]$.  

Let $X$ be a finitely generated ${\hat{R}}$-module.  We recall a construction from  \cite{Eisenbud80} which gives an action of $S$ on $\ext(X,k)$.  Let $(F_\bl,\dell)$ be a free resolution of $X$ over ${\hat{R}}$.  Each $F_n={\hat{R}}^{i_n}$ and we may view $\dell$ as a sequence of matrices with entries in ${\hat{R}}$.  Let $\tilde{F}_n=Q^{i_n}$ and $\tilde{\dell}$ be the lift of $\dell$ to $\tilde{F}_\bl$.  Since $\dell^2=0$, we know that $\tilde{\dell}^2$ is a sequence of matrices whose entries are in the ideal $(f_1,\dots,f_c)$.  Thus we may write
\[\tilde{\dell}^2=\sum_{i=1}^c f_i\tilde{\Phi}_i\]
where $\tilde{\Phi}_i$ is a sequence of matrices with entries in $Q$.  Set $\Phi_i=\tilde{\Phi}_i\tns {\hat{R}}$.  We may now define an action on $\bigoplus_{n=0}^\infty F_n$ by ${\hat{R}}[\chi_1,\dots,\chi_c]$ where we set $\chi_ir=\Phi_i(r)$ for every $r\in F_n$.  This induces an action of ${\hat{R}}[\chi_1,\dots,\chi_c]$ on $\ext(X,k)=\bigoplus_{i=0}^\infty \ext^i(X,k)$.  It is shown in \cite{Eisenbud80} that the operators $\chi_i$ commute  turning $\ext(X,k)$ into a graded $S$-module, where each $\chi_i$ is degree $2$.   Furthermore, Eisenbud shows that this action is independent of  our choices of $F_\bl$ and $\tilde{\Phi}_i$.   Also, $\ext(X,k)$ is actually a finitely generated $S$-module.  The ring $S$ is known as the ring of cohomological operators and has been the focus of much study including  \cite{Avramov89,AvramovBuchweitz00a,AvramovGasharovPeeva97,Eisenbud80,Mehta76}.   In fact, there are several equivalent methods for constructing the action of $S$ on $\ext(X,k)$, the first of which was given in \cite{Gulliksen74}: see \cite{AvramovSun98} for a detailed discussion. 

The following result shows that the action is actually an invariant of the ideal generated by the regular sequence.  

\begin{theorem}[{\cite[Proposition 1.7]{Eisenbud80}, cf$.$ \cite[(3.11)]{AvramovSun98}}]\label{supvar2.1}

Let $f_1,\dots,f_c$ and $f_1',\dots,f_c'$ be two regular sequences of a regular local ring $Q$ which generate the same ideal.  Write
\[f_i=\sum_{j=1}^c q_{i,j}f_j'\]
with each $q_{i,j}\in Q$.  Letting $\chi_1,\dots,\chi_c$ and $\chi_1',\dots,\chi_c'$  be the cohomological operators associated to $f_1,\dots,f_c$ and $f_1',\dots,f_c'$ respectively, we have
\[\chi_j'=\sum_{i=1}^c q_{i,j} \chi_i\]

\end{theorem}

Thus the matrix $(q_{i,j})$ acts as a change of basis matrix, changing the coordinates of $\PPP^c_{\tilde{k}}$.  When $k=\tilde{k}$,  any change of coordinates of $\PPP_{\tilde{k}}^{c-1}$ corresponds to choosing a different regular sequence which generates the ideal $(f_1,\dots,f_c)$.  This important fact is critical to several proofs in this document, thus we state it precisely.

\begin{prop}\label{supvar2.8}

Assume that $k$ is algebraically closed and set $I=(f_1,\dots,f_c)$ where $f_1,\dots,f_c$ is a regular sequence of a regular local ring $Q$.  Let $\ph:\PPP^{c-1}_k\to \PPP^{c-1}_k$ be an automorphism, i.e. a change of coordinates.  Then there exists a regular sequence $f_1',\dots,f_c'$ generating $I$ such that $\ph_*(\chi_i)=\chi_i'$ where $\chi_1,\dots,\chi_c$ and $\chi_1',\dots,\chi_c'$  are the cohomological operators associated to $f_1,\dots,f_c$ and $f_1',\dots,f_c'$ respectively.

\end{prop}

\begin{proof}

Set $\psi=\ph^{-1}$, and let $\tilde{\ph}$ and $\tilde{\psi}$ be the lifts of $\ph$ and $\psi$ in $Q$ such that $\tilde{\phi}=\tilde{\ph}^{-1}$.  We can regard $\tilde{\ph}$ and and $\tilde{\psi}$ as a matrices whose entries are $q_{i,j}\in Q$ and $p_{i,j}\in Q$ respectively.  Set 
\[f_i'=\sum_{i=1}^cp_{i,j} f_j.\]
By Nakayama's lemma, $f_1',\dots,f_c'$ generates $I$, and since there are $c$ elements, $f_1',\dots,f_c'$ is necessarily a regular sequence.    However since $\tilde{\ph}\tilde{\psi}$ is the identity, we also have
\[f_i=\sum_{j=1}^c q_{i,j}f_j'.\]
It follows from Theorem \ref{supvar2.1} that
\[\chi_j'=\sum_{i=1}^c q_{i,j} \chi_i=\ph_*(\chi_j.)\]

\end{proof}

\subsection{Cohomological supports}
With the machinery  of the cohomological operators in place, we may now discuss cohomological supports.  We define 
\[V(Q,\underline{f};X)=\{\bar{a}\in \PPP_{\tilde{k}}^c\mid g(\bar{a})=0 \quad\forall g\in\ann_S\ext(X,k)\}\]
where $\tilde{k}$ is the algebraic closure of $k$.  %This algebraic set is a cone in $\AAA_{\tilde{k}}^c$.   

\begin{definition}
Let $R$ be a complete intersection ring.    Following \cite{AvramovBuchweitz00}, for a finitely generated $R$-module $M$, the {\it cohomological support}, denoted $V^*(M)$, is the projectivization in $\PPP_{\tilde{k}}^{c-1}$ of the cone $V(Q,\underline{f};\hat{M})$.  Occasionally, $V^*_R(M)$ will be used to indicate which ring is used to compute the cohomological support.  
\end{definition}

\begin{prop}[{\cite[Theorem 5.3]{AvramovBuchweitz00}}]

For any finitely generated $R$-module $M$, $V^*(M)$ is independent of the choices of $Q$, and $\underline{f}$ up to a change of coordinates.

\end{prop}

\begin{remark}
What we call the cohomological support is referred to  as the support variety in \cite{AvramovBuchweitz00} and other works.  In \cite{AvramovIyengar07}, the terminology cohomological support and cohomological variety are both used.  Since geometers generally require varieties to be irreducible closed subsets and since $V^*(M)$ is generally not irreducible, we have decided to use the term cohomological support.  
\end{remark}
 
\begin{remark}
In \cite{AvramovBuchweitz00} and in other works, the authors consider $V^*(M)$ as a cone in $\AAA_{\tilde{k}}^c$.  To facilitate the statement of certain results, we have found it easiest to work in projective space.   
\end{remark}

The following is a combination of the results \cite[Theorem 5.6,Theorem 6.1]{AvramovBuchweitz00}.  

\begin{theorem}\label{supvar2.9}

For  finitely generated $R$-modules $M$ and $N$, the following are equivalent.
\begin{enumerate}
\item $V^*(M)\cap V^*(N)=\emptyset$
\item $\tor_{\gg 0}(M,N)=0$
\item $\ext^{\gg 0}(M,N)=0$
\item $\ext^{\gg 0}(N,M)=0$
\end{enumerate}

\end{theorem}

Hence cohomological supports encode homological information about a module. The following result gives another description of cohomological supports.  

\begin{theorem}[{\cite[Theorem 5.2]{AvramovBuchweitz00},\cite[Corollary 3.11]{Avramov89}}]\label{supvar2.2}

Suppose that the residue field $k$ is algebraically closed.  For any module $M\in\md(R)$, we have
\[V^*(M)=\{(a_1,\dots,a_c)\in\PPP^{c-1}_k\mid \pd_{Q/(\tilde{a}_1f_1+\cdots+\tilde{a}_cf_c)} \hat{M}=\infty\}\]
where $\tilde{a}_i$ is a lift in $Q$ of $a_i$.  

\end{theorem}

From this result and Lemma \ref{supvar2.6}, we can easily deduce these corollaries.

\begin{corollary}\label{supvar2.3}

For a finitely generated $R$-module $M$, $V^*(M)=\emptyset$ if and only if $\pd M<\infty$.  Also $V^*(k)=\PPP^{c-1}_{\tilde{k}}$.  

\end{corollary}

\begin{corollary}\label{supvar2.12}

Let  $f_1,\dots,f_c$ be a regular sequence of  a regular local ring $Q$, and let $k[\chi_1,\dots,\chi_c]$ be the ring of cohomological operators for $Q/(f_1,\dots,f_c)$.  Take $n$ such that $1\le n\le c$.  Let $H$ be the linear space defined by $\chi_{n+1}=\cdots =\chi_c=0$.    For any module $M$ over $Q/(f_1,\dots,f_c)$,  
\[V^*_{Q/(f_1,\dots,f_n)}(M)=V^*_{Q/(f_1,\dots,f_c)}(M)\cap H.\]
% where $V^*_{Q/(f_2,\dots,f_c)}(M)$ and $V^*_{Q/(f_1,\dots,f_n)}(M)$ are the cohomological supports of $M$ computed over $Q/(f_2,\dots,f_c)$ and $Q/(f_1,\dots,f_c)$ respectively.  

\end{corollary}

\begin{corollary}\label{supvar2.14}

Suppose $M\in\md(R)$ and $x\in R$ is regular on $M$.  Then $V^*(M)=V^*(M/xM)$.

\end{corollary}

\begin{proof}

Let $\tilde{x}\in Q$ be a lift of $x$.  Then $\tilde{x}$ is still regular on $M$ and so $M$ has finite projective dimension over $Q/a$ if and only if $M/\tilde{x}M=M/xM$ does too.  The result now follows from Theorem \ref{supvar2.2}.

\end{proof}

A generalization of Corollary \ref{supvar2.3} exists involving complexity.

\begin{definition}

For a sequence $(a_n)_{n\ge 0}$ of nonnegative integers, we can define the {\it complexity} 
\[\cx (a_n)_{n\ge 0}=\min\{\deg f\mid f\in \QQ[t] \quad a_n\le f(n)\quad  \forall n\gg 0\}+1.\]
For a module $M$, we set $\cx M=\cx \be_n(M)$.  

\end{definition}

A module has finite projective dimension if and only if $\cx M=0$.  Since $R$ is a complete intersection of codimension $c$, $\cx k$ equals $c$.  

\begin{prop}[{ \cite[Theorem 5.6]{AvramovBuchweitz00}}]\label{2.14}

For any $R$-module, we have $\dim V^*(M)=\cx M-1$

\end{prop}

\begin{remark}

Note that the previous result considers $V^*(M)$ as a closed set of projective space instead of a cone in affine space.

\end{remark}

The following are useful results on cohomological supports.

\begin{theorem}[{\cite[Corollary 2.3]{Bergh07},\cite[Theorem 7.8]{AvramovIyengar07}}]\label{supvar2.11}

If $k$ is algebraically closed, for each closed set $V\sbe \PPP^{c-1}_k$, there is a maximal Cohen-Macaulay module $M$ such that $V^*(M)=V$.

\end{theorem}

When working with cohomological supports, it is important to be able to reduce to the case where $R$ is complete and $k$ is algebraically closed.  These two results which allow us to do this.  

\begin{lemma}\label{supvar2.7}

For any $R$-module $M$, we have $V_R^*(M)=V^*_{\hat{R}}(\hat{M})$.  

\end{lemma}

\begin{lemma}[{ \cite[Lemma 2.2]{AvramovBuchweitz00},\cite[App., Th\'{e}or\`{e}m 1, Corollaire]{Bourbaki83})}]\label{supvar2.6}

There exists a local complete intersection ring $(\tilde{R},\tilde{\mm},\tilde{k})$ of codimension $c$ such that $\tilde{R}$ is a flat extension of $R$, $\mm\tilde{R}=\tilde{m}$, and the induced map $k\to\tilde{k}$ is the inclusion of $k$ into its algebraic closure.  Furthermore, we have $V_R^*(M)=V_{\tilde{R}}^*(M\tns \tilde{R})$.  

\end{lemma}

\subsection{Thick subcategories}
There is a deep connection between cohomological supports and the thick subcategories of $\md(R)$.  This connection begins with the following result.  
\begin{prop}[{\cite[Theorem 5.6]{AvramovBuchweitz00}}]\label{supvar2.4}

If \[0\to X_1\to X_2\to X_3\to 0\] is exact, then 
\[V^*(X_i)\sbe V^*(X_j)\cup V^*(X_l)\]
with $\{i,j,l\}=\{1,2,3\}$. In particular, $V^*(M)=V^*(\zz M)$.   Furthermore
\[V^*(X\oplus Y)=V^*(X)\cup V^*(Y)\]

\end{prop}

\begin{definition}

A subcategory $\C\sbe \md(R)$ is thick if 
\begin{enumerate}
\item $R\in\C$
\item $\C$ is closed under direct summands, that is if $X\oplus Y\in C$ then $X,Y\in \C$
\item $\C$ has the two out of three property, that is if $0\to X_1\to X_2\to X_3\to 0$ and $X_i,X_j\in\C$, then $X_l\in \C$ with $\{i,j,l\}=\{1,2,3\}$.
\end{enumerate}
Let $\thick M$ denote the smallest thick category containing $M$.

\end{definition}

The thick subcategories of $R$ are in bijection with the thick subcategories of the triangulated category $\underline{\MCM(R)}$, the stable category of maximal Cohen-Macaulay modules.  The category of modules of finite projective dimension is  thick. We can generalize this example:
by Proposition \ref{supvar2.4}, for any $V\sbe\PPP^{c-1}_{\tilde{k}}$, the category
\[\{M\in\md(R)\mid V^*(M)\sbe V\}\]
is thick.  It turns out that we can use the cohomological supports to classify all the thick subcategories of $\md(R)$ in this manner.  

Before proceeding, we will fix some notation which will prevent us from confusing the geometric subtleties regarding cohomological supports.  We will use the symbol $\PPP^{c-1}_{\tilde{k}}$ to denote the closed points of the scheme $\proj\tilde{k}[\chi_1,\dots,\chi_c]$.  Since 
\[V^*(M)\sbe \PPP^{c-1}_{\tilde{k}}=\maxspec\proj\tilde{k}[\chi_1,\dots,\chi_c]\]
we may let $\bar{V}^*(M)$ be the closure of $V^*(M)$ in $\proj\tilde{k}[\chi_1,\dots,\chi_c]$.  The integral extension $k[\chi_1,\dots,\chi_c]\sbe\tilde{k}[\chi_1,\cdots,\chi_c]$  induces the projection $\pi:\proj \tilde{k}[\chi_1,\dots,\chi_c]\to \proj k[\chi_1,\dots,\chi_c]$.  Set $\tilde{V}^*(M)=\pi(\bar{V}^*(M))$.  The following result is from \cite{Stevenson12}, and, in the zero dimensional case, from \cite[Remark 5.12]{CarlsonIyengar12}.  

\begin{theorem}

If $R$ is an isolated singularity and a complete intersection, we have a bijection
\[ \{\mbox{Thick subcategories of } \md(R)\} \longleftrightarrow \{\mbox{Specialization closed subsets of } \proj k[\chi_1,\dots,\chi_c]\}\]
given by the maps
\[\M\longmapsto \bigcup_{M\in \M} \tilde{V}^*(M) \]
\[\{M\in\md(R)\mid \tilde{V}^*(M)\sbe U\}\reflectbox{$\longmapsto$} U\]
where $\M$ is a thick subcategory of $\md(R)$ and $U$ is a specialization closed subset of $\proj k[\chi_1,\dots,\chi_c]$.  Furthermore, the closed subsets of $\proj k[\chi_1,\dots,\chi_c]$ correspond to the  cyclic thick subcategories, i.e. the thick subcategories of the form $\thick M$ for some $M\in\md(R)$.

%hen the thick subcategories of $\md(R)$ are in bijection with the specialization closed subsets of the scheme $\proj k[\chi_1,\dots,\chi_c]$. 
% Furthermore, the cyclic thick subcategories, i.e. the thick subcategories of the form $\thick M$, are in bijection with the closed sets of  the scheme $ \proj k[\chi_1,\dots,\chi_c]$.This bijection is given by $\thick M\mapsto \overline{V^*(M)}\cap \proj k[\chi_1,\dots,\chi_c]$.  When $k$ is algebraically closed, the  cyclic thick subcategories are in bijection with the closed subsets of $\PPP^{c-1}_k$.  

\end{theorem}

\begin{remark}\label{supvar2.10}

In \cite{Stevenson12}, Stevenson actually classifies the thick subcategories over an arbitrary complete intersection ring in terms of the specialization closed subsets of a scheme $Y=\sng \proj Q[y_1,\dots,y_c]/(f_1y_1+\cdots+f_cy_c)$ where $y_i$ are in determinants over $Q$ and $\sng$ denotes the singularities.  He does this by assigning  $M\in\md(R)$ a closed set
%How do we know this set is closed?
 in $Y$ which we will denote by $\Ss (M)$.  On a set theoretic level, we have
\[Y=\coprod_{p\in \sing R} \proj k(p)[\chi_1,\dots,\chi_{c(p)}]\]
where $c(p)$ is the codimension of $R_p$ and $k(p)$ is the residue field of $R_p$.  Furthermore, we have
\[\Ss(M)\cap \proj k(p)[\chi_1,\dots,\chi_{c(p)}]=\tilde{V}_{R_p}^*(M_p).\]

\end{remark}

The next proposition follows immediately from these remarks.

\begin{prop}

For two modules $M,N\in\md(R)$, $V^*_{R_p}(M_p)\sbe V^*_{R_p}(N_p)$ for every $p\in\sing R$ if and only if $\thick M\sbe \thick N$.  

\end{prop}

The following example illustrates the different between working in $\PPP^{c-1}_{\tilde{k}}$ and $\proj\tilde{k}[\chi_1,\dots,\chi_c]$.

\begin{example}

Assume that $k$ is algebraically closed and let $\M=\{M\in \md(R)\mid \cx M\le R\}$.  It is easy to check that $\M$ is thick.  For each module $M\in\M$, $\dim V^*(M)$ is 0 or $-1$ and hence is a point or empty.  Furthermore,  for each point $p\in  \PPP^{c-1}_k$, there is module $M$ such that $V^*(M)=p$ by Theorem \ref{supvar2.11}.  It follows that $M\in \M$.  Also, we have
\[U:=\bigcup_{M\in\M} V^*(M)= \PPP^{c-1}_k\]
Now $V^*(k)=U$.  But, unless $R$ is a hypersurface, $\cx k>1$.  We can see that it does not follow that $k$ is in $\M$.  Indeed, we have $\tilde{V}^*(k)$, the support of $k$ on the entire scheme, is all of $\proj k[\chi_1,\dots,\chi_c]$.  On the other hand, he have
\[\tilde{U}:=\bigcup_{M\in\M} \tilde{V}^*(M)= \PPP^{c-1}_k\bigcup_{M\in\M} V^*(M)=\maxspec\proj k[\chi_1,\dots,\chi_c].\]
So the support of $\M$ on the entire scheme is the specialization closed set consisting of the closed points.  Thus when understanding thick subcategories, one needs to consider the entire scheme.

\end{example}

\subsection{Geometric join}\label{supvarjigglyopo}

In this subsection we give attention to another construction central to this paper.  

\begin{definition}

Let $U,V\sbe \PPP^n_k$ be Zariski closed sets.  We define the {\it join} of two sets to be 
\[\join(U,V)=\overline{\bigcup_{u\in U \ v\in V\ u\ne v} \lin(u,v)}\]
where $\lin(u,v)$ is the projective line containing $u$ and $v$.  Furthermore, in the case when $U=V$, we set $\sec V=\join(V,V)$ which, when $V$ is a variety, we refer to as the secant variety of $V$.  

\end{definition}

\begin{remark}

When $U$ and $V$ are disjoint Zariski closed sets, we may simplify this definition to 
\[\join(U,V)=\bigcup_{u\in U \ v\in V} \lin(u,v).\]
and we still obtain a closed set, \cite[Proposition 6.13,Example 6.14]{Harris95}.  In most contexts in this paper, we will be taking the join of disjoint sets.  

\end{remark}

\begin{remark}\label{supvar2.13}

There is some ambiguity with this definition when $V$ is empty.  To that end, we establish the following convention:
\[\join(U,\emptyset)=U.\]
This is the convention followed in \cite{Adlandsvik87}.  We justify this convention by working in affince space: the empty set corresponds to the the zero point and the join of the cone and the zero point is simply the original cone.

%Later we will be taking the joins of cohomological supports.  Let $(R,\mm,k)$ be complete intersection.  For notational convenience, suppose $R$ is complete and $k$ is algebraically closed.   Since $\pd 0=0$, we have $V^*(0)=\emptyset$.  Despite this, for any $R$-module $M$, we will establish the following additional convention:
%\[\join(V^*(M),V^*(0))=\emptyset\]
%We can justify these conventions using this convention by working in affine space.  Here, if $N\ne 0$ and $\pd N<\infty$, then the support of $\ext(N,k)$ over $k[\chi_1,\dots,\chi_c]$ is just the maximal ideal $(\chi_1,\dots,\chi_c)$.   Thus, thinking of $V^*(M)$ an $V^*(N)$ as cones in affine space, we have
%\[\join(V^*(M),V^*(N))=\join(V^*(M),0)=V^*(M).\]
%This intuition informs our first convention.  However, when $N=0$, the support of $\ext(0,k)=0$ over over $k[\chi_1,\dots,\chi_c]$ is truly the empty set.  So in affine space, we have $\join((V^*M),V^*(0))=\emptyset$.  This intuition informs our second convention.  

\end{remark}

The following in anopther interesting fact about joins.  

\begin{lemma}[{ \cite[Proposition 6.13,Example 6.14]{Harris95}}]\label{2.15}

If $U$ and $V$ are irreducible closed sets, then $\join(U,V)$ is also irreducible.

\end{lemma}

To visualize the join, consider the following easy examples.  The join of two distinct points is a projective line, and the join of two skew lines is a three dimensional projective linear space.  In fact, the join of any two linear spaces is the smallest linear space containing both of them.  In particular, the secant variety of any linear space is itself.  %The join is not always linear: the join of a point and circle (not containing the point) is a double cone.  

\begin{theorem}[{\cite[1.1]{Adlandsvik87}}]\label{2.13}

For two closed sets $U,V\sbe \PPP^n_k$, we have
\[\dim \join(U,V)\le \dim U+\dim V+1\]
and if $U\cap V=\emptyset$, then 
\[\dim \join(U,V)= \dim U+\dim V+1.\]

\end{theorem}

The converse is not true, and, in fact, when it is not known in general when $\dim \join(U,V)= \dim U+\dim V+1$ in the case $U\cap V\ne \emptyset$ . In particular, an active topic of research is understanding when $\join(V,V)\ne 2\dim V+1$.

\section{Joins of cohomological supports}\label{supvarjoinssection}

In this section, let $(R,\mm,k)$ be a local complete intersection of codimension $c$.  The goal of this section is to prove the main result of this paper, namely the following theorem.    

\begin{theorem}\label{supvarmain1}

Let $R$ be a local complete intersection, and $M,N\in\md(R)$ with $M,N\ne 0$.  If $\tor_{>0}(M,N)=0$, then $V^*(M\tns N)=\join(V^*(M),V^*(N))$.  

\end{theorem}

 We prove Theorem \ref{supvarmain1} in a few steps.  First we recall a critical fact.

\begin{lemma}[{\cite[Proposition 2.1]{Miller98}}]\label{1.1.5}

If $\tor_{>0}(M,N)=0$, then $\cx(M\tns N)=\cx M+ \cx N$.  

\end{lemma}

We now prove a special case of the main result.

\begin{lemma}\label{supvar1.4}

If $R$ has codimension two and $\tor_{>0}(M,N)=0$ with $M,N\ne 0$, then 
\[V^*(M\tns N)=\join(V^*(M),V^*(N)).\]

\end{lemma}

\begin{proof}

By Lemma \ref{supvar2.7} and Lemma \ref{supvar2.6}, we may assume that $R$ is complete and $k$ is algebraically closed.  If $\pd M,\pd N<\infty$, then $\pd M\tns N<\infty$ and the conclusion is clear.  

Assume that $\pd M=\pd N=\infty$.  Since $V^*(M)$ and $V^*(N)$ are disjoint, nonempty, and lie in $\PPP_k^1$,  we know that $\dim V^*(M)=\dim V^*(N)=0$.  Therefore, the complexity of both $M$ and $N$ is one.  Since $\tor_{>0}(M,N)=0$, the complexity of $M\tns N$ is two, and thus $\dim V^*(M\tns N)=1$.  This means that $V^*(M\tns N)=\PPP_k^1= \join(V^*(M),V^*(N))$.

We now assume that $\pd M<\infty$ and $\pd N=\infty$. Using the conventions in Remark \ref{supvar2.14}, we may assume that $M\ne 0$.  We wish to show that $V^*(M\tns N)=V^*(N)$. Letting 
\[0\to R^{n_t}\to \cdots \to R^{n_0}\to M\to 0\]
be a free resolution, the sequence
\[0\to N^{n_t}\to \cdots \to N^{n_0}\to M\tns N\to 0\]
is exact, implying $V^*(M\tns N)\sbe V^*(N)$.  By Lemma \ref{1.1.5}, we have 
\[\cx(M\tns N)=\cx M+\cx N=\cx N.\]  Therefore $\dim V^*(M\tns N)$ is the same as $\dim V^*(N)$.  So if $V^*(N)$ is irreducible, we are done.  In particular, if  $\dim V^*(N)=1$, then $V^*(N)=\PPP_k^1$ and we are done.  

So suppose $\dim V^*(N)=0$, that is $V^*(N)=\{p_1,\dots ,p_n\}$ where $p_i$ are points.  The short exact sequence $0\to \zz N\to R^m\to N\to 0$ yields the short exact sequence  
\[0\to M\tns \zz N\to M^m\to M\tns N\to 0.\]
But since $V^*(M)$ is empty, we have $V^*(M\tns \zz N)=V^*(M\tns N)$ by Proposition \ref{supvar2.4}.  Thus, since $V^*(N)=V^*(\zz N)$ we may assume that $N$ is maximal Cohen-Macaulay by replacing $N$ with a sufficiently high syzygy.  Therefore, by Theorem  3.1 of \cite{Bergh07}, we may write $N=N_1\oplus\cdots\oplus N_n$ with $V^*(N_i)=\{p_i\}$.  Since each $V^*(N_i)$ is irreducible, we have
\begin{align*}
V^*(M\tns N)=V^*(&M\tns N_1\oplus\cdots \oplus M\tns N_n)\\
&=V^*(M\tns N_1)\cup\cdots \cup V^*(M\tns N_n)=\{p_1\}\cup\cdots\cup\{p_n\}=V^*(N)
\end{align*}
which completes the proof.  

\end{proof} 

\begin{prop}\label{supvar1.3}

  If $\tor_{>0}(M,N)=0$ with $M,N\ne 0$, then $\join(V^*(M),V^*(N))\sbe V^*(M\tns N)$.   

\end{prop}

\begin{proof}

By Lemma \ref{supvar2.7} and Lemma \ref{supvar2.6}, we may assume that $R$ is complete and $k$ is algebraically closed.  We proceed by induction on the codimension of $R$, which we will denote by $c$. When $c=1$, the cohomological supports lie in $\PPP_k^0$, a point.  Thus the cohomological support of a module is either that point or empty, depending on whether or not the module has finite projective dimension.  Since $\tor_{>0}(M,N)=0$, the result follows from the equality $\cx M+\cx N=\cx M\tns N$. If $c=2$, the statement is true by Lemma \ref{supvar1.4}.

Now suppose that $c\ge 2$.  It suffices to show that for any hyperplane $H\sbe\mathbb{P}^{c-1}_k$ we have 
\[\join(V^*(M),V^*(N))\cap H\sbe V^*(M\tns N)\cap H\]
Since $c\ge 2$, any hyperplane is a linear space with dimension at least one.  Therefore, for any $x\in V^*(M)\cap H$ and $y\in V^*(N)$, the projective line between $x$ and $y$ is also in $H$.  Because of this, we have 
\[\join(V^*(M),V^*(N))\cap H=\join(V^*(M)\cap H,V^*(N)\cap H).\]
Thus we need to show that 
\[\join(V^*(M)\cap H,V^*(N)\cap H)\sbe V^*(M\tns N)\cap H.\]

To that end, we fix a hyperplane $H$.  Now may write $R=Q/(f_1,\dots,f_c)$ where $Q$ is a regular local ring and $f_1,\dots,f_c$ is a regular sequence.  By Proposition \ref{supvar2.8}, we may change our coordinate system and assume that $H=V(\chi_1)$ where $k[\chi_1,\dots,\chi_c]$ is the ring of cohomological operators.  Set $T=Q/(f_2,\dots,f_c)$ and $f=f_1$.  Note that $T$ is a complete intersection with $\codim T=c-1$, $f$ is regular on $T$, and $R=T/(f)$.  For any module $X\in\md(R)$, $V_R^*(X)\cap H=V_T^*(X)$ by Corollary \ref{supvar2.12}.  Therefore we need to show that 
\[\join(V_T^*(M),V_T^*(N))\sbe V_T^*(M\tns N).\]

Since $\tor^R_{>0}(M,N)=0$, by \cite[Lemma 9.3.8]{Avramov10}, we have $\tor_1^T(M,N)=M\tns N$ and $\tor_{>1}^T(M,N)=0$.  It follows that $\tor_{>0}^T(M,\zz_T N)=0$.  After tensoring $0\to \zz_T N\to T^t\to N\to 0$ with $M$, we get the exact sequence
\[0\to M\tns N\to M\tns \zz_T N\to M^t\to M\tns N\to 0.\]
Thus, by induction,  we have
\[ \join(V_T^*(M),V_T^*(N))=\join(V_T^*(M),V_T^*(\zz_T N))\sbe V_T^*(M\tns \zz_T N )\sbe V_T^*(M\tns N) \cup V_T^*(M).\]
Note, that we can only choose $\zz_T N$ to not be zero.  A similar argument using $\zz_T M$ gives us 
\[\join (V_T^*(M),V_T^*(N)) \sbe V_T^*(M\tns N)\cup V_T^*(N).\] 
This implies that 
\[\join(V_T^*(M),V_T^*(N)) \sbe V_T^*(M\tns N)\cup (V_T^*(M)\cap V_T^*(N))=V_T^*(M\tns N)\]
proving the claim.

\end{proof}

\begin{lemma}\label{supvar1.2}

Suppose that $c=\codim R\ge 2$, $R$ is complete, and $k$ is algebraically closed.  Fix  $M\in\md(R)$ such that $V_R^*(M)=q$ for some point $q\in \PPP_k^{c-1}$.  For any $p\in \PPP_k^{c-1}$ distinct from $q$, there exists an $L\in\MCM$ such that $V^*_R(L)=p$ and $V^*_R(M\tns L)=\join(p,q)$. 

\end{lemma}

\begin{proof}

As $R$ is complete, we write $R=Q/(f_1,\dots,f_c)$ where $Q$ is a regular local ring and $\underline{f}$ is a regular sequence. After a change of coordinates, we may assume that $p=(1,0,0,\dots,0)$.  Set $T=Q/(f_1)$, $X=\zz^{d-1}_T k$ where $d=\dim Q$, and $L=X/(f_2,\dots,f_c)X$.   We prove that $L$ is our desired module.

First, we show that  $V^*(L)=p$.    Since $\pd_T L=\infty$, it follows from Theorem  \ref{supvar2.2} that $p\in V^*(L)$.  Take any point $p'\in \PPP^{c-1}_k$ such that $p'\ne p$.  By Theorem \ref{supvar2.11}, there exists a $Y\in\md(R)$ such that $p'= V^*(Y)$.  It follows from Theorem \ref{supvar2.2} that  $\pd_T Y<\infty$.  Furthermore,  we have $\tor_i^{R}(Y,L)\cong \tor_i^{T}(Y,X)$  for all $i>0$, and thus $\tor_{\gg0}^{R}(Y,L)=0$.  Therefore $V^*(Y)\cap V^*(L)=\emptyset$ and so $p'\notin V^*(L)$.  We now have $V^*(X)=\{p\}$ as claimed.  

Set $l=\join(p,q)$.  We now show that $V^*(M\tns L)=l$.    By proposition \ref{supvar1.3}, we have $l\sbe V^*(M\tns L)$.  Take any point $r\notin l$.  We claim that $r$ is not in $V^*(M\tns L)$. After changing coordinates, we may assume that $r=(0,1,0,\dots, 0)$.  Set $S=Q/(f_1,f_2)$ and $X'=X/f_2X$.  Let $l'$  be the projective line defined by $p$ and $r$.  Since $r\notin l$, we have $q\notin l'$. Hence Corollary \ref{supvar2.12} implies that $V_{S}^*(M)=V^*_R(M)\cap l'=\emptyset$.   Therefore, $\tor_{\gg 0}^{R'}(M,X')=0$.  However, since $X'$ is maximal Cohen-Macaulay over $S$, Lemma \ref{supvar3.4} implies that $\tor_{>0}^{S}(M,X')=0$.  Since $M\tns L\cong M\tns X'$, and since $\codim S=2$, Lemma \ref{supvar1.4} implies that
\[V_{S}^*(M\tns L)=V_{S}^*(M\tns X')=\join(V_{S}^*(M),V_{S}^*(X'))=V_S^*(X').\]
By Corollary \ref{supvar2.12} and Corollary \ref{supvar2.14}, we have $V_{S}^*(X')=V^*_{S}(L)=V_R^*(L)\cap l'=p$.  Therefore, we have $p=V_{S}^*(M\tns L)=V_R^*(M\tns L)\cap l'$ by Corollary \ref{supvar2.12}.   Since $r$ is in $l'$, $r$ is not in $V^*(M\tns L)$, as desired.  

\end{proof}

We now proceed with the proof of the main result of this paper.  

\begin{proof}[Proof of Theorem \ref{supvarmain1}]

By Lemma \ref{supvar2.7} and Lemma \ref{supvar2.6}, we may assume that $R$ is complete and $k$ is algebraically closed.   First, we note that we may assume that neither $M$ nor $N$ is zero since otherwise the statement is trivial.  Proposition \ref{supvar1.3} gives us one containment, which leaves us to show the reverse containment: 
\[V^*(M\tns N)\sbe \join(V^*(M),V^*(N)).\]
We will proceed by induction using induction on $\al(M,N)=2\depth R-\depth M-\depth N$.  
Assume for the moment that we have shown the base case, i$.$ e$.$ the theorem is true when when $\al(M,N)=0$, which is precisely the case when both $M$ and $N$ are maximal Cohen-Macaulay.  Suppose that $\al(M,N)>0$.  Then one of the modules, say $M$, is not maximal Cohen-Macaulay, and  $\al(\zz M,N)=\al(M,N)-1$.  Tensoring the short exact sequence $0\to \zz M\to R^s\to M\to 0$ with $N$ yields 
\[0\to  \zz M\tns N\to N^s\to M\tns N\to 0.\]  By induction, we have the following
\[V^*(M\tns N)\sbe V^*(N)\cup V^*(\zz M\tns N)\sbe V^*(N)\cup \join(V^*(\zz M),V^*(N))=\join(V^*(M),V^*(N))\]
which yields the desired inclusion.

We now prove the base case.  Assume that $\al(M,N)=0$ or, equivalently, that $M$ and $N$ are maximal Cohen-Macaulay modules.   
%We will first show the theorem when both $V^*(M)$ and $V^*(N)$ are irreducible.  Since $\tor_{>0}(M,N)=0$, we have $\cx(M\tns N)=\cx(M)+\cx(N)$ by Lemma \ref{1.1.5}.  It follows from Theorem \ref{2.13} and Proposition \ref{2.14} that $\dim V^*(M\tns N)=\dim\join(V^*(M),V^*(N))$.  Since  $\join(V^*(M),V^*(N))\sbe V^*(M\tns N)$ by Proposition \ref{supvar1.3}, and since $\join(V^*(M),V^*(N))$ is irreducible by Lemma \ref{2.15}, it follows that $\join(V^*(M),V^*(N))= V^*(M\tns N)$.
%
%We now remove the assumption that $V^*(M)$ and $V^*(N)$ are irreducible.  Since $M$ and $N$ are maximal Cohen-Macaulay, by \cite[Theorem 3.1]{Bergh07}, we can write $M=M_1\oplus\cdots \oplus M_m$ and $N=N_1\oplus \cdots N_n$ such that each $V^*(M_i)$ and $V^*(N_j)$ is irreducible. Note that $\tor_{>0}(M_i,N_j)=0$ for each $i$ and $j$.  Using the results of the previous paragraph, we have
%\begin{align*}
%\join(V^*(M),V^*(N))
%	&=\join\left(V^*(M_1)\cup\cdots \cup V^*(M_m),V^*(N_1)\cup\cdots \cup V^*(N_n)\right)\\
%	&=\bigcup_{i,j} \join(V^*(M_i),V^*(N_j))\\
%	&= \bigcup_{i,j} V^*(M_i\tns N_j)\\
%	&=V^*\left(\bigoplus_{i,j} M_i\tns N_j\right)\\
%	&= V^*(M\tns N)\\
%\end{align*}
%which completes the proof.  
First we show the theorem when $V^*(M)$ is simply a single point, say $q$.  Take any  $p\notin \join(V^*(M),V^*(N))$.
  By Lemma \ref{supvar1.2},  there exists maximal Cohen-Macaulay module $L$ such that $V^*(L)=p$ and $V^*(M\tns L)=\join(p,q)=\join(V^*(M),V^*(L))$.   However, since $p\notin\join(V^*(M),V^*(N))$, there are no lines containing $p$, $q$ and a point in $V^*(N)$.  Therefore $V^*(M\tns L)=\join(p,q)$ and $V^*(N)$ are disjoint.  Since $M,N,L$ are all maximal Cohen-Macaulay, this shows that $\tor_{>0}(M,L)=0$ and also $\tor_{>0}(M\tns L, N)=0$.  Now let $A_\bl,B_\bl,C_\bl$ be free resolutions of $L,M,N$ respectively.  Then,  $(M\tns L)\tns C_\bl$ is quasi-isomorphic to $\mbox{Tot}_\bl(A_\bl\tns B_\bl\tns C_\bl)$ which is quasi-isomorphic  to $A_\bl\tns(M\tns N)$.  Therefore  $\tor_i(L, M\tns N)\cong \tor_i(M\tns L,N)=0$ for all $i\gg 0$.  This implies that $V^*(M\tns N)$ does not contain $p=V^*(L)$, yielding the desired containment.

Now we show the general case.  Again, take a point $p\notin\join(V^*(M),V^*(N))$.    By Theorem \ref{supvar2.11} there exists an $L\in\md(R)$ with $V^*(L)=p$.  In the previous paragraph, we have shown that $V^*(M\tns L)=\join(V^*(M),V^*(L))$.  Thus the argument  in the previous paragraph still applies, completing the proof.

\end{proof}

\section{Corollaries of Theorem \ref{supvarmain1}}\label{supvarcorollaries}

We now state some interesting corollaries of Theorem \ref{supvarmain1}.   The following is immediate.  

\begin{corollary}\label{supvar4.1}

If $N$ is not zero and $\tor_{>0}(M,N)=0$, then $V^*(M)\sbe V^*(M\tns N)$.  

\end{corollary}

From this we derive a plethora of other corollaries. 

\begin{corollary}

If $N\ne 0$ and $\tor_{>0}(M,N)=0$, then the following hold.  
\begin{enumerate}
\item $\ext^{\gg 0}(M\tns N,L)=0\Rightarrow\ext^{\gg 0}(M,L)=0$
\item $\ext^{\gg 0}(L,M\tns N)=0\Rightarrow\ext^{\gg 0}(L,M)=0$
\item $\tor_{\gg 0}(M\tns N,L)=0\Rightarrow\tor_{\gg 0}(M,L)=0$
\end{enumerate}

\end{corollary}

\begin{proof}

The previous corollary shows that $V^*(M\tns N,L)=\emptyset$ implies that $V^*(M,L)=\emptyset$.

\end{proof}

\begin{corollary}

Suppose $\tor_{>0}(M,N)=0$.  If $\sing R\sbe \supp N\cup (\spec R\del \supp M)$, then $M$ is in $\thick M\tns N$.  In particular, if $R$ is an isolated singularity, then  $\tor_{>0}(M,N)=0$ implies that $M,N\sbe \thick M\tns N$ when $M,N\ne 0$. 

\end{corollary}

\begin{proof} 

First note that $\tor_{>0}(M_p,N_p)=0$ for every $p\in\spec R$. Let $p\in\sing R$.  Then either $p\in \supp N$ or $p\notin\supp M$.  Then Corollary  \ref{supvar4.1} implies that $V_{R_p}^*(M_p)\sbe V_{R_p}^*((M\tns N)_p)$ for all $p\in\sing R$.  The result then follows by Remark \ref{supvar2.10}.  

\end{proof}

%\begin{corollary}
%
%Let $\F$ denote the free locus of a module.  If $\tor_{>0}(M,N)=0$, then  $\F(M\tns N)\cup \sup (N)$ is contained in $\F(M)$.  
%
%\end{corollary}
%
%\begin{proof}
%
%Suppose $(M\tns N)_p$ is free, and $N_p$ is not zero.  Since $\tor_{>0}(M_p,N_p)=0$, we have 
%\[V^*_{R_p}(M_p)\sbe V^*_{R_p}(M_p\tns N_p)=\emptyset.\]
%Hence $M_p$ has finite projective dimension.   However, the depth formula implies that 
%\[2\depth R_p=\depth (M\tns N)_p+\depth R_p=\depth M_p+\depth N_p\]
%and therefore $\depth M_p=\depth R_p$ and so $M_p$ is free.  
%
%\end{proof}
%
%\begin{remark}
%
%Note that the proceeding Corollary holds in a more general setting.  Letting $\X_c$ be the maximal Cohen-Macaulay modules of complexity $c$, the corollary holds for if we replace the free locus with the $\X_c$ locus.  
%
%\end{remark}

The next results use the following lemma.

\begin{lemma}\label{supvar3.4}

Suppose $R$ is a complete intersection ring and either $M$ or $N$ is maximal Cohen-Macaualay.  If $\tor_{\gg 0}(M,N)=0$, then $\tor_{>0}(M,N)=0$. 
\end{lemma}
  
\begin{proof}

If not, then for every $n\in \NN$, there exists a cosyzygy $M'$ of $M$ such that $\tor_{n}(M',N)\ne 0$ and $\tor_{\gg 0}(M',N)=0$.  But this contradicts the fact that complete intersections are AB; see \cite[Corollary 3.4]{{HunekeJorgensen03}}.

\end{proof}

We may also use Theorem \ref{supvarmain1} to construct modules with linear cohomological supports. 

\begin{corollary}

Assume that $k$ is algebraically closed and $R$ is complete.  Set $p_i=(0,\dots,1,\dots,0)\in\PPP^{c-1}_k$ be the point that is one in the $i$th position and zeros elsewhere.  Let $L$ be the affine span of $p_1,\dots,p_n$.  Set 
\[X_i=\frac{\zz^{d-1}_{Q/(f_i)} k}{(\zz^{d-1}_{Q/(f_i)} k)(f_1,\dots,\hat{f_i},\dots, f_c)}\]
 where $d=\dim Q$.  Then $X_1\tns\cdots \tns X_n$ is maximal Cohen-Macaulay and $L=V^*(X_1\tns\cdots \tns X_n)$.

\end{corollary}

Note that by Proposition \ref{supvar2.8}, after changing coordinate any linear space of $\PPP_k^{c-1}$ is of the form of $L=V^*(X_1\tns\cdots \tns X_n)$.  

\begin{proof}

By Theorem \ref{supvar2.2}, $V^*(X_i)=\{p_i\}$.  We work by induction on $n$.  When $n=1$, we are done.  So assume the statement is true for $n-1$.   Let $L'$ be the affine span of $p_1,\dots,p_{n-1}$.  The induction hypothesis implies that $V^*(X_{n-1})=L'$.    Since $X_n$ is maximal Cohen-Macaulay ,  $\tor_{>0}(X_1\tns\cdots \tns X_{n-1}, X_n)=0$ by Lemma \ref{supvar3.4}.  Then $X_1\tns\cdots \tns X_n$ is maximal Cohen-Macaulay and by Theorem \ref{supvarmain1}, we have
\[V^*(X_1\tns\cdots \tns X_n)=\join(V^*(X_1\tns\cdots \tns X_{n-1}),V^*(X_n))=\join(L',p_n)=L\]
proving the claim.

\end{proof}

The main result of this paper also prevents certain tor modules from vanishing.  

\begin{corollary}

Suppose  $M_1,\dots,M_{c+1}$ are nonfree maximal Cohen-Macaulay modules.  Then for some $i\in \{1,\dots, c\}$, 
\[\tor_n(M_1\tns\cdots\tns M_i,M_{i+1})\ne 0\]
 for infinitely many $n$.  

\end{corollary}

\begin{proof}

Proceeding by contradiction, suppose that 
\[\tor_{\gg 0}(M_1\tns\cdots\tns M_i,M_{i+1})=0\]
for each $1\le i\le c$.  Inducting on $i$, we will show two statements: first that
\[V^*(M_1\tns\cdots\tns M_i)=\join(V^*(M_1),\dots,V^*(M_i))\] 
for each $i$ in $\{1,\dots,c\}$, and second that $V^*(M_1\tns \dots\tns M_i)$ contains a linear space of dimension $i-1$.  When $i=1$, the statement is trivial.  So suppose the statement is true for $i$.  Since $R$ is a complete intersection and each $M_{i+1}$ is maximal Cohen-Macaulay, Lemma \ref{supvar3.4} implies that $\tor_{>0}(M_1\tns\cdots\tns M_i,M_{i+1})=0$.  By Theorem \ref{supvarmain1}, it follows that 
\begin{align*}
V^*(M_1\tns\cdots\tns M_i\tns M_{i+1})&=\join((V^*(M_1),\dots,V^*(M_i)),V^*(M_{i+1}))\\
&=\join(V^*(M_1),\dots,V^*(M_{i+1}))
\end{align*}
Furthermore, let $L$ be the dimension $i$ linear space in $V^*(M_1\tns\cdots\tns M_i)$ guaranteed by the induction hypothesis.  Take $x\in V^*(M_{i+1})$ which exists since $M_{i+1}$ is not free.  Now $x$ is not in $L$ and so 
\[\join(L,x)\sbe V^*(M_1\tns\cdots\tns M_{i+1}).\]  But $\join(L,x)$ is a linear space of dimension $i+1$, proving the claim.  

Now the contradiction is clear, for there is a $c$-dimensional linear space contained in $V^*(M_1\tns\cdots\tns M_{c+1})$ which is a closed subset of $\PPP^{c-1}$.

\end{proof}

It turns out that we can prove an analogue of Theorem \ref{supvarmain1} using Ext.

\begin{theorem}\label{main4}

Suppose $R$ is a complete intersection ring and $M,N\in\md(R)$.  If $\ext^{>0}(M,N)=0$, then $V^*(\hm(M,N))=\join(V^*(M),V^*(N))$.

\end{theorem}

\begin{proof}

Since $\ext^{>0}(M,N)=0$, \cite[Lemma 2.5]{Araya12} implies that $M$ is maximal Cohen-Macaulay.   By \cite[Proposition 3.6]{Hartshorne98} and \cite[2.1.1]{Jorgensen08}, we have the following exact sequence
\[\tor_2(\tr M,N)\to \hm(M,R)\tns N\to \hm(M,N)\to \tor_1(\tr M,N)\to 0\]
where $\tr M$ is the Auslander-Bridger transpose of $M$.  However, since $R$ is Gorenstein, $\tr M$ is an inverse syzygy of $M$, i.e. there exists an exact sequence
\[0\to M\to F\to G\to \tr M\to 0\]
with $F,G$ free.  So $\tor^{\gg 0}(\tr M,N)=0$.  Since $\tr M$ is again Cohen-Macaulay, it follows from Lemma \ref{supvar3.4} that $\tor_{>0}(\tr M,N)=0$.  Therefore, the above short exact sequence gives us the isomorphism $M^*\tns N\cong \hm(M,N)$.  

Since $V^*(M)=V^*(M^*)$, we know that $\tor_{\gg 0}(M^*,N)=0$.  Since $M$ is maximal Cohen-Macaulay, Lemma \ref{supvar3.4} implies $\tor_{>0}(M^*,N)=0$.  Therefore, by Theorem \ref{supvarmain1} we have
\[V^*(\hm(M,N))=V^*(M^*\tns N)=\join(V^*(M^*),V^*(N))=\join(V^*(M),V^*(N)).\]

\end{proof}

This result provides an elementary proof of the following result.

\begin{corollary}

If $\ext^{>0}(M,N)=0$, then $\cx\hm(M,N)=\cx M+\cx N$.

\end{corollary}

\begin{proof}

It follows from the previous theorem, Theorem \ref{2.13}, and Proposition \ref{2.14} that
\[\cx\hm(M,N)=\dim \join(V^*(M),V^*(N))+1=\dim V^*(M)+\dim V^*(N) +2=\cx M+\cx N.\]

\end{proof}

\section{Vanishing of Tor}\label{supvarvanishing}

In this section, we prove the following theorem which gives sufficient conditions for the vanishing of Tor modules. The novelty here is that these conditions involve dimension instead of depth.  We then apply this result to Theorem \ref{supvarmain1}.

\begin{theorem}\label{supvarmain3}

Suppose $(R,\mm,k)$ is an AB ring and $M$ and $N$ are Cohen-Macaulay (but not necessarily maximal Cohen-Macaulay).  Then if $\tor_{\gg 0}(M,N)=0$ and 
\[\dim M\tns N+\dim R\le \dim M+\dim N\]
then $\tor_{>0}(M,N)=0$.    

\end{theorem}

 Although we are primarily interested in complete intersection rings, we prove the theorem in terms of AB rings for the sake of generality. Since complete intersection rings are AB, the result specializes to the desired context.   We refer the reader to  \cite{HunekeJorgensen03} for the definition and basic facts about AB rings.  The only fact about AB rings we use is the following.  

\begin{lemma}\label{3.2}

Suppose $R$ is AB.  If $\tor_{\gg 0}(M,N)=0$, then we have $\tor_{>\dim R}(M,N)=0$.

\end{lemma}

The key to the proof of Theorem \ref{supvarmain3}, is the following observation.

\begin{lemma}

Suppose $(R,\mm,k)$ is local and $M$, $N$, and $R$ are Cohen-Macaulay.  If 
\[\dim M\tns N+\dim R\le \dim M+\dim N\]
 then there exists a sequence $x_1,\dots,x_n\in \ann N$  which is regular on $M$ and $R$ such that $N$ is maximal Cohen-Macaulay over $R/(x_1,\dots,x_n)$.

\end{lemma}

\begin{proof}

We work by induction on $r=\dim R-\dim N$.  When $r=0$, the statement is trivial. So suppose $r>0$. We divide the proof into two cases.  First, suppose there is an $x\in \ann N$  which is regular on $R$ and $M$.  Since $\dim R/xR=\dim R-1$ and $\dim M/xM=\dim M-1$, we have $\dim (M/xM)\tns N+\dim R/xR\le\dim M/xM+\dim N$.  Thus by induction, there exists a regular sequence $x_2,\dots,x_n$ on $M/xM$ and $R/xR$ such that $N$ is maximal Cohen-Macaulay over $R/(x,x_2,\dots,x_n)$.  Thus $x,x_2,\dots, x_n$ is our desired regular sequence.  

For the second case, suppose there is no  $x\in \ann N$  which is regular on $R$ and $M$. In other words, suppose that 
\[\ann N\sbe \left(\bigcup_{p\in \ass M} p\right)\cup\left(\bigcup_{q\in\ass N} q\right).\]
  Then  $\ann N\sbe p$ with either $p\in \ass R$ or $p\in \ass M$.  If $p\in \ass R=\min R$, then $\dim N=\dim R$, and the empty regular sequence suffices.  So suppose $p\in \ass M$.  Then $p$ is in $\supp M\cap \supp N=\supp M\tns N$, thus $\dim M\tns N\ge \dim R-\h p$.  But since $M$ is Cohen-Macaulay, $\dim M=\dim R -\h p\le \dim M\tns N$.  Since $\dim M\tns N+\dim R\le \dim M+\dim N$, this means that $\dim R\le \dim N$, which implies that $N$ is maximal Cohen-Macaulay, proving the claim.

\end{proof}

\begin{proof}[Proof of Theorem \ref{supvarmain3}]

By the previous lemma, we may let $x_1,\dots, x_n\in\ann N$ be a regular sequence on $M$ and $R$ such that $N$ is maximal Cohen-Macaulay over $R/(x_1,\dots,x_n)$.  Set $\bar{R}=R/(x_1,\dots,x_n)$. Since $\bar{R}$ is Gorenstein, there exists a long exact sequence of the form
\[0\to N\to \bar{R}^{m_0}\xto{\dell^0} \bar{R}^{m_1}\xto{\dell^1} \bar{R}^{m_2}\xto{\dell^2}\cdots\]
Set $N^i=\zz^{-i}_{\bar{R}} N=\ker \dell^i$.   Note that $N^0=N$.   Now $\tor^R_{>0}(M,\bar{R})=0$ since $x_1,\dots,x_n$ is regular on $M$.  Thus, considering the short exact sequence $0\to N^i\to \bar{R}^m\to N^{i+1}\to 0$, we see that $\tor_j(M,N^i)\cong \tor_{j+1}(M,N^{i+1})$. In particular, $\tor_{\gg0}(M,N^i)=0$ for all $i$.  However, Lemma \ref{3.2} implies that $\tor^R_{>\dim R}(M,N^i)=0$ for all $i$.  Thus for all $j>0$, we have 
\[\tor^R_j(M,N)\cong \tor^R_{j+\dim R+1}(M,N^{\dim R+1})=0\] completing the proof.  

\end{proof}

Theorem  \ref{supvarmain3} and Theorem \ref{supvarmain1} give this immediate corollary.

\begin{corollary}

Suppose $R$ is a local complete intersection and $M$ and $N$ are Cohen-Macaulay (but not necessarily maximal Cohen-Macaulay). Then if $V^*(M)\cap V^*(N)=\emptyset$ and $\dim M\tns N+\dim R\le \dim M+\dim N$, then $V^*(M\tns N)=\join(V^*(M),V^*(N))$.  

\end{corollary}

This corollary gives a relation between the actual support of a module and also the cohomological support.  

\begin{remark}

Theorem  \ref{supvarmain3} can actually be stated in more general terms using AB dimension, which is defined by Araya in \cite[Definition 2.2]{Araya12}.  The application of Lemma \ref{3.2} is the only place in the proof where the AB assumption is used.  However, if we assume that $R$ is Cohen-Macaulay and $\ABdim N<\infty$, then arguments in  \cite[Proposition 3.2, Theorem 3.3]{HunekeJorgensen03} show that the conclusion of Lemma \ref{3.2} still holds.  Therefore, Theorem  \ref{supvarmain3} still holds if we only assume that $R$ is Cohen-Macaulay and $\ABdim N<\infty$.

\end{remark}

\section{Homological Conjectures}

\subsection{Conjectures both new and old}

In this section we introduce some new conjectures. The main motivation comes from Section 5 and some of the famous homological conjectures in commutative algebra. 

\begin{definition}
\ \\
\begin{itemize}
\item[{\bf Strong DI}] The strong dimension inequality holds if $\tor_{\gg 0}(M,N)=0$ implies \[\dim M+\dim N\le \dim R+\dim M\tns N\]  
\item[{\bf DE}] The dimension formula holds, if $\tor_{>0}(M,N)=0$ implies \[\dim M+\dim N=\dim R+\dim M\tns N\]
\item[{\bf Para}] The parameter conjecture holds if $\pd N<\infty$ and $p\in\min R$ implies there exists a $q\in \min N$ such that $p\sbe q$.  This is equivalent to saying that if $x$ is a parameter element on $N$, then it is a parameter element on $R$.
\end{itemize}

\end{definition}

These statements seem to be unstudied.  If $N$ has no embedded primes, then {\bf Para} follows from the zero divisor conjecture.  Serre's intersection conjectures are related to {\bf Strong DI} and {\bf DE}.  Let $\ld$ denote the length function.  

\begin{conjecture}
Suppose $\pd N<\infty$ and $\ld(M\tns N)<\infty$ and set
\[\chi(M,N)=\sum_{i\ge 0} (-1)^i\ld(\tor_i(M,N))\]
The following statements were conjectured by Serre.
\begin{itemize}
\item[{\bf DI}] $\dim M+\dim N\le\dim R$
\item[{\bf Non-negativity}] $\chi(M,N)\ge 0$
\item[{\bf Vanishing}] $\chi(M,N)=0$ whenever $\dim M+\dim N<\dim R$
\item[{\bf Positivity}] $\chi(M,N)>0$ whenever $\dim M+\dim N=\dim R$
\end{itemize}
\end{conjecture}

Serre proved each of these conjectures over unramified regular local rings and showed that {\bf DI} holds for all regular local rings.  For arbitrary regular local rings, {\bf Non-negativity} and {\bf vanishing} were established in \cite{Gabber95} and \cite{GilletSoule85,Roberts85} respectively.  However, {\bf Positivity} is still open in the ramified case.  For singular rings,  {\bf Non-negativity}, {\bf Vanishing}, and {\bf Positivity} are false in general, see \cite{DuttaHochsterMcLaughlin85}.

Obviously,  {\bf DI} is a special case of {\bf Strong DI}.   When  $\tor_{>0}(M,N)=0$, then $\chi(M,N)>0$.  Thus  when $\pd N<\infty$ and $\ld(M\tns N)<\infty$, then {\bf DE} is implied by  Serre's intersection conjectures.  Although {\bf Vanishing}, and {\bf Positivity} are not true in general, no counterexamples to {\bf Strong DI} and {\bf DE} are known to the authors.  The following conjectures are also relevant.

\begin{definition}
\ \\
\begin{itemize}

\item[{\bf PS}] Peskine and Spiro conjectured in \cite{PeskineSzpiro73} that   if $\pd N<\infty$ and $\ld(M\tns N)<\infty$, then \[\dim M\le \grade N\]
\item[{\bf GC}] The grade conjecture asserts that $\grade N+\dim N=\dim R$.  
\item[{\bf Small MAC}] The small Cohen-Macaulay conjecture asserts that every ring $R$ has a finitely generated maximal Cohen-Macaulay module.  

\end{itemize}

\end{definition}

We now list the relations between these conjectures.

\begin{prop}\label{implications}

\begin{enumerate}
\item\label{2} {\bf PS} implies {\bf DI}.
\item\label{3} If {\bf GC} holds, then {\bf DI} implies {\bf PS}.
\item\label{7} Equidimensional implies {\bf GC}.
\item\label{6} Suppose holds {\bf Strong DI} holds when $\pd N\le \infty$.  Then  {\bf Para} holds too.
\item\label{6.5} Suppose $R$ is equidimensional and catenary.  If {\bf PS} holds at every localization of $R$, {\bf Para} holds as well. 
\item\label{8} Suppose {\bf Small MAC} holds for all rings.  If for every $p\in \spec R$  {\bf DE} holds and $R_p$ is equidimensional, then we have {\bf Para}.
\end{enumerate}

When $R$ is equidimensional and catenary, we summarize implications (1)-(\ref{6.5}) with the following diagram.
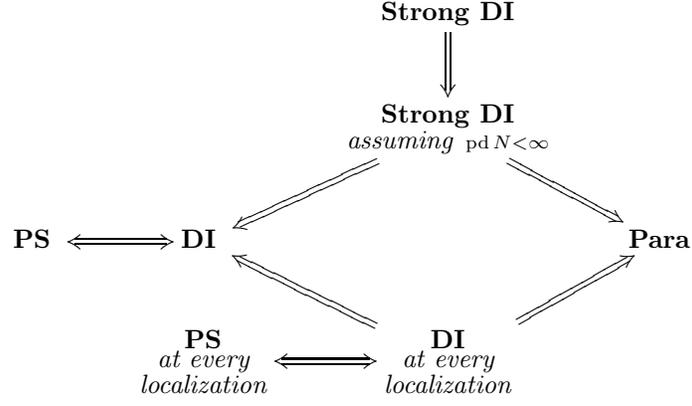
\[\xymatrix{
						&					& \mbox{{\bf Strong DI}} \ar@{=>}[d]													& \\
						&					& {\substack{\mbox{{\bf Strong DI}} \\  \mbox{assuming } \pd N<\infty}} \ar@{=>}[dl] \ar@{=>}[dr]  	&\\
\mbox{{\bf PS} } \ar@{<=>}[r]	& \mbox{{\bf DI} } 		&																			& \mbox{{\bf Para}} \\		
						&{\substack{ \mbox{{\bf PS}} \\ \mbox{at every}\\ \mbox{localization}}} \ar@{<=>}[r] & {\substack{ \mbox{{\bf DI}} \\ \mbox{at every}\\ \mbox{localization}}} \ar@{=>}[ul] \ar@{=>}[ur] & \\														
}\]
\end{prop}

The authors would like to thank Melvin Hochster for bringing (\ref{6.5}) to their attention.  

\begin{proof}

Statement (\ref{2}) follows from the inequality $\grade(N)+\dim N\le \dim R$.  Furthermore, {\bf GC} asserts this inequality is an equality, yielding (\ref{3}).  Statement (\ref{7}) is \cite[Theorem 3.6]{Beder14}.

We now show (\ref{6}).  Let $p\in \min R$.  Take $q\in \min N/pN$ such that $\dim R/q=\dim N/pN$.  Since $\pd N<\infty$, then  {\bf Strong Dimeq} implies
\[\dim N+\dim R/p\le \dim R+\dim N/pN=\dim R+\dim R/q.\]
Since $\dim R=\dim R/p$, this inequality implies $\dim N\le \dim R/q$.  However, since $q$ contains $\ann N$, this implies that $\dim N= \dim R/q$.  It follows that $q$ is a minimal prime of $N$.  

To show (\ref{6.5}), assume that $\pd N<\infty$ and let $p\in\min R$.   Let $q$ be minimal over $p+\ann N$.   Now $\lambda(N_q\tns (R/p)_q)<\infty$, and so if {\bf PS} holds at every localization then $\dim (R/p)_q\le \grade N_q$.  Since $R$ is catenary and equidimensional then $\dim (R/p)_q=\dim R_q$.  Therefore, $\dim R_q\le \grade R_q$, and so $R_q$ is Cohen-Macaulay and $\ann N$ is $qR_q$-primary.  Therefore, $q\in \min N$ as desired.

It remains to prove (\ref{8}).  Let $p\in \min R$.  Let $q$ be minimal over $p+\ann N$.   By assumption, there exists a small maximal Cohen-Macaulay module  $M$ over $(R/p)_q$.  Now since $R$ is equidimensional at every localization, $M$ is also maximal Cohen-Macaulay over $R_q$.  Since $\pd N_q<\infty$, it is not difficult to show that $\tor^{R_q}_{>0}(M,N_q)=0$. Hence, {\bf DE} implies
\[\dim M+\dim N_q=\dim R_q+\dim M\tns N_q=\dim R_q\]
and so we have $\ld(N_q)<\infty$.  This implies that $q$ is in $\min N$.

\end{proof}

\subsection{Examples}

We now list some cases where these conjectures hold.

\begin{prop}

When $R$ is a regular local ring, $R$ satisfies {\bf Strong DI}.  When $R$ is an unramified regular local ring,  $R$  satisfies {\bf DE}.

\end{prop}

\begin{proof}

Take $p\in\min M$ and $q\in\min N$ such that $\dim M=\dim R/p$ and $\dim N=\dim R/q$.  Furthermore, choose a $\pi\in\spec R$ that is minimal over $p+q$.  Suppose that $\pi=\mm$.  Then since $R$ is regular and $\dim R/p\tns R/q=0$, we have
\[\dim M+\dim N=\dim R/p+\dim R/q\le \dim R\le \dim R+\dim M\tns N.\]
So assume that $\pi\ne \mm$.  Since $p,q\sbe \pi$,  we have $\dim M_\pi+\dim R/\pi=\dim M$ and likewise for $N$.  By induction, we have the desired inequality:
\[\dim M+\dim N=\dim M_\pi+\dim N_\pi-2\dim R/\pi\le \dim R_\pi+\dim (M\tns N)_\pi-2\dim R/\pi\le\dim R+\dim M\tns N.\]

Now we prove the second statement. Suppose that $\tor_{>0}(M,N)=0$ and $R$ is unramified.  By {\bf Strong DI}, it suffices to show that $\dim R +\dim M\tns N\le \dim M+\dim N$.   Take $q\in \spec R$ such that $\dim R/q=\dim M\tns N$.  Then $\dim (M\tns N)q=0$ and $\chi(M_q,N_q)=\lambda(M\tns N)_q>0$.  Therefore Serre's intersection conjectures tell us that $\dim R_q=\dim M_q+\dim N_q$.  Localizing at $q$ yields
\[\dim R +\dim M\tns N=\dim R_q+2\dim R/q=\dim M_q+\dim N_q+2\dim R/q\le \dim M+\dim N\]
which yields the desired result.

\end{proof}

\begin{prop}

Suppose $R$ is a standard graded $k$-algebra with $k$ a field.  Then {\bf Strong DI} and {\bf DE} hold over graded modules.

\end{prop}

\begin{proof}

Let $M$ and $N$ be graded modules.  Set 
\[\chi(M,N)(t)=\sum_{i\ge 0} (-1)^i H_{\tor_i(M,N)}(t)\]
where $H_{\tor_i(M,N)}(t)$ is the Hilbert series of $\tor_i(M,N)$. We know from \cite[Lemma 7]{AvramovBuchweitz93} that $\chi(M,N)(t)H_R(t)=H_M(t)H_N(t)$.  When $\tor_{\gg 0}(M,N)=0$, $\chi(M,N)(t)$ is a rational function, and so we let $d$ denote the degree of its denominator.  It follows that $d+\dim R=\dim M+\dim N$.  When $\tor_{> 0}(M,N)=0$, then $\chi(M,N)(t)= H_{M\tns N}(t)$ and so $d=\dim M\tns N$.  This gives us {\bf DE}.  When $\tor_{\gg 0}(M,N)=0$, we have $d\le \max_i \dim \tor_i(M,N)\le \dim M\tns N$ where the second inequality comes from the inclusion $\supp \tor_i(M,N)\sbe \supp M\tns N$.  This yields {\bf Strong DI}.  

\end{proof}

In the situation of the previous proposition, whenever $\tor_{\mbox{odd}}(M,N)=0$ and $\tor_{\gg 0}(M,N)=0$, we have $\dim M+\dim N=\dim R+\dim M\tns N$.  Indeed in this circumstance, the degree of the denominator of $\chi(M,N)(t)$ is just $\dim \bigoplus_i\tor_i(M,N)=\dim M\tns N$, and then proof of of the proposition still applies.  

\subsection{Applications}

These conjectures have some interesting consequences.

\begin{prop}

Suppose {\bf Para} holds.  If $p\in\spec R$ and there exists an $R$-module $N$ such that $\pd N<\infty$ and $\supp N=V(p)$, then $p$ contains every minimal prime of $R$.  

\end{prop}

\begin{proof}

Since $\min N=\{p\}$, the result is immediate.

\end{proof}

\begin{prop}

Let $S$ be  an $R$-algebra which is finitely generated as an $R$ module.  Suppose {\bf DE} holds for both $R$ and $S$.  If  $\tor^R_{> 0}(M,N)=0$ and $\tor^S_{> 0}(M,N)=0$ for $S$-modules $M$ and $N$, then $\dim R=\dim S$.

\end{prop}

\begin{proof}

The result follows from comparing $\dim M+\dim N=\dim R+\dim M\tns N$ and $\dim M+\dim N=\dim S+\dim M\tns N$.

\end{proof}

The New Intersection Theorem states that if $\pd N<\infty$ then $\dim M\le \pd N+\dim M\tns N$. We can use {\bf Strong DI} to give a similar inequality for Cohen-Macaulay rings.  However, before proceeding we need a definition.

\begin{definition}

When $R$ is Cohen-Macaulay, set $\mbox{CM-}\dim N=\depth R-\dim N$.

\end{definition}

Actually, $\mbox{CM-}\dim$ is a well studied homological dimension which can be defined even for non-Cohen-Macaulay rings.  Since our attention is restricted to the Cohen-Macaulay case, we do not give the full definition, but refer the reader to \cite{Gerko01}.

\begin{prop}

Suppose $R$ is Cohen-Macaulay and satisfies {\bf Strong DI}.  Then if $\tor_{\gg 0}(M,N)=0$, we have $\dim M\le \mbox{CM-}\dim N+\dim M\tns N$.

\end{prop}

\begin{proof}

The assumptions imply
\[\depth N+\dim M\le \dim N+\dim M\le \dim R+\dim M\tns N=\depth R+\dim M\tns N\]
which tells us that 
\[\dim M\le \depth R-\depth N+\dim M\tns N=\mbox{CM-}\dim N+\dim M\tns N.\]

\end{proof}

Recall that $S_n$ is the collection of modules such that  $\depth M_p\ge \min\{n,\h p\}$.

\begin{prop}

Suppose $(R,\mm)$ is a local Cohen-Macaulay and the depth formula holds at every localization.  Furthermore, suppose that either
\begin{enumerate}
\item {\bf Strong DI} holds when one of the modules has finite projective dimension
\item {\bf DI} holds at every localization  
\end{enumerate}
Assume $N\ne 0$ and that there exists a module $L$ such that $\pd L\le \infty$ and $\supp L=N$.   Then if $M\tns N$ satisfies $S_n$, we have $\{p\in\supp M\mid \h p< n\}\sbe \supp N$.

\end{prop}

\begin{remark}

When $\pd N\le \infty$,  \cite[Theorem 3.1]{CelikbasPiepmeyer14} states that $M$ satisfies $S_n$ too.  When $\supp M\sbe \supp N$, this result is trivial by the depth formula.  Hence,  \cite[Theorem 3.1]{CelikbasPiepmeyer14} is actually a statement about the behaviour of $\depth M_p$ for $p$ in $\supp M\del\supp N$.  In this respect, our results are similar.  In fact, the arguments are similar too.  

\end{remark}

\begin{proof}

Take $p\in \supp M$ with $\h p< n$. Choose a $q\in \min L\tns R/p$ such that $\dim R/q=\dim L\tns R/p$.  Note that $p\in\supp N=\supp L$.  We claim that $\dim N_q+\dim (R/p)_q\le \dim R_q$.  Suppose {\bf DI} holds  for $R_q$.  Then since $\dim L_q\tns(R/p)_q=0$, we have
\[\dim N_q+\dim (R/p)_q=\dim L_p+\dim (R/p)_q\le \dim R_q.\]
Now suppose {\bf DI} holds when one of the modules has finite projective dimension.  Then we have the following
\begin{align*}
\dim N_q+\dim (R/p)_q\le \dim N+&\dim R/p-2\dim R/q\\
	&=\dim L+\dim R/p-2\dim R/q\le \dim R +\dim L\tns R/p-2\dim R/q=\dim R_q
\end{align*}
which establishes the claim.

We now make the following computation.  The first inequality is \cite[A.6.2]{Christensen00}.
\begin{align*}
\h p	&\ge \depth M_p \\
	&\ge  \depth M_q-\dim (R/p)_q\\
	&\ge  \depth M_q+\dim N_q-\dim R_q\\
	&\ge \depth M_q+\depth N_q-\depth R_q\\
	& = \depth (N\tns M)_q\\
	& \ge\min\{n, \h q\} 
\end{align*}
Since $\h p< n$ by assumption, we have $\h p\ge \h q$ which forces $p=q$.  Therefore, $p\in\supp N$ as desired.

\end{proof}

%We close with a few further observations on $S_n$. Recall  $\tilde{S}_n$  the modules  satisfying $\depth M_p\ge \min\{n,\dim M_p\}$
%\begin{prop}
%
%Let $x\in R$, $(R,\mm)$, be  local, and suppose $\tor_{> 0}(M,R/x)=0$.
%
%\begin{enumerate}
%\item $M/xM\in S_n$ implies that $M\in S_{n-1}$
%\item If $x$ regular on $M$, then  $M/xM\in\tilde{S}_n$ implies $M\in \tilde{S}_n$
%\item If $R$ satisfies either {\bf Strong DI} or {\bf DE} and $\dim R=\dim R/x$, then $M/xM\in \tilde{S}_n$ implies that $M\in \tilde{S}_{n-1}$
%\end{enumerate}
%
%\end{prop}
%
%\begin{proof}
%
%Let $p\in\supp M$.  If $p\in \supp M/xM$, the statements are easy to verify.  So suppose that $(M/xM)_p=0$.  It follows that $x\notin p$.  Let $q$ be minimal over $x+p$.  Localizing at $q$ allows us to assume that $x+p$ is $\mm$ primary.  Therefore we have $\depth M_p\ge \depth M-1$.  Then statement (1) follows from 
%\[\depth M_p\ge \depth M-1\ge \depth M/xM -1\ge\min\{n,\h \mm\}-1\ge \min\{n-1,\h p\}\]
%
%To show statement (2), we have 
%\[\depth M_p\ge \depth M-1= \depth M/xM\ge  \min\{n,\dim M/xM\}=\min\{n,\dim M-1\}\ge \min\{n,\dim M_p\}\]
%
%The assumptions of statement (3), shows that $\dim M+\dim R/xR\le \dim R+\dim M/xM$ which implies that $\dim M=\dim M/xM$.  Therefore, we have
%\[\depth M_p\ge \depth M-1\ge \depth M/xM -1\ge\min\{n-1,\dim M/xM-1\}=\min\{n-1,\dim M-1\}\ge \min\{n-1,\dim M_p\}\]
%
%\end{proof}

\section{Questions and examples}\label{supvarstupidquestions}

What happens to Theorem \ref{supvarmain1} if we remove the assumption that all the Tor modules vanish?  The following two examples show that in general neither containment holds.  

\begin{example}\label{supvarEx1}

Let $k$ be a field and set $R=k[x,y]/(xy)$.  Now the modules $M=R/(x+y)$ and $N=R/(x-y)$ have finite projective dimension.  However, we have
\[\join(V^*(M),V^*(N))=\emptyset\nsupseteq V^*(M\tns N)=\PPP_{\tilde{k}}^0\]
showing that $V^*(M\tns N)$ is not always contained in $\join(V^*(M),V^*(N)$, even if the $\tor_{\gg 0}(M,N)=0$.  

\end{example}  

\begin{example}\label{supvarEx2}

Set $R=\QQ[a,b,c]/(a^2-b^2,b^3-c^3)$ and
\[M=\coker\left[\begin{matrix} 
8ab^2c^2+4abc^3+6b^2c^3+8ac^4+6bc^4+c^5 & 3ab+4b^2+7ac+7bc+4c^2 \\ 4ab^2c^2+6abc^3+9b^2c^3+ac^4+9bc^4+4c^5 & 4ab+5b^2+3ac+5bc+5c^2\\
 \end{matrix}\right]\]
 \[N=\frac{R}{(8ab^2c+8b^2c^2+6ac^3+5bc^3+c^4,3ab+2b^2+3ac+2bc+9c^2)}.\]
 An easy computation in Macaulay2 shows that 
 \[\cx M=0\quad\quad\quad \cx N=2\quad\quad\quad \cx M\tns N=1.\]
 This shows that $\join(V^*(M),V^*(N))=V^*(N)\nsubseteq V^*(M\tns N)$.

\end{example}

We now give an example where none of the modules involved have finite projective dimension.

\begin{example}\label{supvarEx3.5}

Set $R=\QQ[a,b,c,d]/(a^2-b^2,b^2-c^2,d^2)$ and define the ideal
\[I=\left(\frac{3}{5}a+\frac{8}{7}b+\frac{5}{2}c,2a+\frac{1}{2}b+3c,d\right).\]
A simple computation in Macaulay2 shows that
\[V^*(I)=V(3740x_1+477x_2)\]
\[V^*(I\tns I)=V(0)=\PPP_{\tilde{\QQ}}^2\]
Were $\tilde{\QQ}[x_1,x_2,x_3]$ is the ring of cohomological operators over the algebraic closure of $\QQ$.  Since $V^*(I)$ is linear, we have
\[\join(V^*(I),V^*(I))=V^*(I)\nsupseteq V^*(I\tns I).\]

\end{example}

\begin{example}\label{supvarEx3}

Let $R=\QQ[a,b,c]/(a^2,b^2,c^2)$ and $I=(b)$ and $J=(ab)$.  An easy computation yields $V^*(R/I)=V(x_1,x_3)$ and $V^*(R/J)=V(x_1)$ where  $\tilde{\QQ}[x_1,x_2,x_3]$ is the ring of cohomological operators over the algebraic closure of $\QQ$.  Now because $V^*(R/J)$ is a linear space containing $V^*(R/I)$, we have
\[\join(V^*(R/I),V^*(R/J))=V^*(R/J)\nsubseteq V^*(R/J\tns R/I)=V^*(R/(I+J))=V^*(I).\]

\end{example}

The authors wondered if there was a relation between the stable behavior of $V^*(\tor_i(M,N))$ and $\join(V^*(M),V^*(N))$. Investigations using Macaulay2 compelled them to ask the following questions.

\begin{question}\label{supvarQ1}

Does 
\[\bigcup_{i=0}^n V^*(\tor_i(M,N))\]
stabilize as $n$ tends to infinity?

\end{question}

\begin{question}\label{supvarQ2} 

For any modules $M$ and $N$, do we have the following?
\[\join(V^*(M),V^*(N))\sbe \bigcup_{i=0}^\infty V^*(\tor_i(M,N))\]

\end{question}

\begin{prop}

Questions \ref{supvarQ1} and \ref{supvarQ2} are true when $\tor_{\gg 0}(M,N)=0$.

\end{prop}

\begin{proof}

The first question is trivially true in this case.  We prove that the second question is true using induction on the minimal $n$ such that $\tor_{>n}(M,N)=0$.  When $n=0$, the the statement follows from Theorem \ref{supvarmain1}.  So suppose $n>0$.  Then we have $\tor_{>n-1}(\zz M,N)=0$ and so by induction, we have
\begin{align*}
\join(V^*(M),V^*(N))=&\join(V^*(\zz M),V^*(N))\\
&\sbe \bigcup_{i=0}^\infty V^*(\tor_i(\zz M,N))= \bigcup_{i=2}^\infty V^*(\tor_i(M,N))\cup V^*(\zz M\tns N).
\end{align*}
Note that  $M$ is not free and so $\zz M$ is not zero.  The short exact sequence 
\[0\to \zz M\to R^m\to M\to 0\]
 yields
\[0\to \tor_1(M,N)\to \zz M\tns N \to N^m\to M\tns N\to 0.\]
It follows that $V^*(\zz M\tns N)\sbe V^*(N)\cup V^*(M\tns N)\cup V^*(\tor_1(M,N))$ and hence
\[\join(V^*(M),V^*(N))\sbe  \bigcup_{i=0}^\infty V^*(\tor_i(M,N))\cup V^*(N).\]
Similarly, we have 
\[\join(V^*(M),V^*(N))\sbe  \bigcup_{i=0}^\infty V^*(\tor_i(M,N))\cup V^*(M)\]
but since $V^*(M)\cap V^*(N)=\emptyset$, this shows the desired result.

\end{proof}

Question \ref{supvarQ1} is also true when $R$ is a hypersurface, because over such rings $\tor_i(M,N)$ is eventually periodic.  The following example shows that even over a hypersurface  $V^*(\tor_i(M,N))$  does not necessarily stabilize.   

\begin{example}\label{supvarEx4}

Let $k$ be a field and set $R=k[x,y]/(xy)$.  It is easy to show that  $\tor_{\mbox{odd}}(R/(x),R/(y))=0$ and $\tor_{\mbox{even}}(R/(x),R/(y))\cong k$.  The projective dimension of the former is obviously finite, and the projective dimension of the latter is infinite.  Thus $V^*(\tor_i(R/(x),R/(y)))$ cannot stabilize.

\end{example}

Note that Example \ref{supvarEx3.5} shows that we cannot hope to replace the containment in Question \ref{supvarQ2} with equality.  We now show some potential applications.

\begin{prop}\label{supvar5.1}

Suppose Question \ref{supvarQ1} is true for a pair of modules $M$ and $N$.  If we have 
\[\join(V^*(M),V^*(N))=\bigcup_{i=0}^\infty V^*(\tor_i(M,N))\]
then we have the inequality 
\[\max_i\{\cx \tor_i(M,N)\}\le\cx M+\cx N\le \cx \{ a_n= \sum_{i+j=n} \beta_j(\tor_i(M,N))\}.\] 

\end{prop}

\begin{proof}

If we assume that Question \ref{supvarQ1} holds, then the assumptions imply \[
\dim \join(V^*(M),V^*(N))=\max\{\dim V^*(\tor_i(M,N))\}.\]  
This yields the inequality
\[\max\{\cx \tor_i(M,N)\}-1=\max\{\dim V^*(\tor_i(M,N))\}\le \dim V^*(M)+V^*(N)+1=\cx M+\cx N-1\]
It suffices to show the second inequality.  By  \cite{Miller98},  we have 
\[\cx \sum_{i+j=n} \beta_i(M)\beta_j(N)=\cx M+\cx N,\] hence the second inequality follows from the next lemma, Lemma \ref{supvar5.2}.  

\end{proof}

\begin{lemma}\label{supvar5.2}

We have the inequality 
\[\sum_{i+j=n} \beta_i(M)\beta_j(N)\le \sum_{i+j=n} \beta_j(\tor_i(M,N)).\]

\end{lemma}

\begin{proof}

Let $F_\bl$ and $G_\bl$ be minimal free resolutions of $M$ and $N$ and $P_\bl$ a  free resolution of $k$.  Let $T_\bl$ be the total complex $F_\bl\tns G_\bl$.  Let $E^0$ be the double complex $T\tns P_\bl$.  Computing the spectral sequence using the vertical filtration gives $E^1_{i,0}=T_i\tns k$ and $E^1_{i,j}=0$ for $j\ne 0$.  Since $F_\bl$ and $G_\bl$ are minimal, the differential of $T$ is given by  matrices whose entries lie in $\mm$.  Therefore the differential of $E^1$ is zero, and so the spectral sequence collapses. Computing the spectral sequence using the horizontal filtration gives $E^1_{i,j}=\tor_i(M,N)\tns P_j$ and $E^2_{i,j}=\tor_j(\tor_i(M,N),k)$.  We thus have
\[k^{\be_j(\tor_i(M,N))}\implies \bigoplus_{i'+j'=i+j} k^{\be_{i'}(M)\beta_{j'}(N)}.\]
The desired inequality follows since $E^\infty_n$ is a graded vector space whose associated graded space is a quotient of a subspace of $\bigoplus_{i+j=n} k^{\be_j(\tor_i(M,N))}$.  

\end{proof}

For closed sets $U,V\in\PPP_k^n$, it is known that $\dim \join(U,V)\le \dim U+\dim V+1$.  It is not known when precisely this equality is strict.  This question is particularly interesting when $U=V$ and has been the subject of much research. As the following shows, Question 1 and Question 2  are actually related to this topic.

\begin{prop}

Suppose Question 1 and Question 2 have positive answers.  Then for any modules $M$ and $N$ over a complete intersection ring,
\[\dim \join(V^*(M),V^*(N)\le \max_{i\in \NN} \cx \tor_i(M,N)-1.\]

\end{prop}

\begin{proof}

The result is obvious after recalling that $\dim V^*(\tor_i(M,N))=\cx \tor_i(M,N)-1$.  

\end{proof}

\bibliographystyle{amsplain}
%\bibliography{/Users/filiuspelei/aamath/Communications/BibliographyII}
%\bibliography{/Users/williats/aamath/communications/BibliographyII}
\bibliography{BibliographyII}

\providecommand{\bysame}{\leavevmode\hbox to3em{\hrulefill}\thinspace}
\providecommand{\MR}{\relax\ifhmode\unskip\space\fi MR }
\providecommand{\MRhref}[2]{%
  \href{http://www.ams.org/mathscinet-getitem?mr=#1}{#2}
}
\providecommand{\href}[2]{#2}
\begin{thebibliography}{10}

\bibitem{Adlandsvik87}
Bj{\o}rn {\AA}dlandsvik, \emph{Joins and higher secant varieties}, Math. Scand.
  \textbf{61} (1987), no.~2, 213--222. \MR{947474 (89j:14030)}

\bibitem{Araya12}
Tokuji Araya, \emph{A homological dimension related to {A}{B} rings},  (2012),
  arxiv:1204.4513v1.

\bibitem{Avramov89}
L.~L. Avramov, \emph{Modules of finite virtual projective dimension}, Invent.
  Math. \textbf{96} (1989), no.~1, 71--101. \MR{981738 (90g:13027)}

\bibitem{Avramov10}
Luchezar~L. Avramov, \emph{Infinite free resolutions [mr1648664]}, Six lectures
  on commutative algebra, Mod. Birkh\"auser Class., Birkh\"auser Verlag, Basel,
  2010, pp.~1--118. \MR{2641236}

\bibitem{AvramovBuchweitz93}
Luchezar~L. Avramov and Ragnar-Olaf Buchweitz, \emph{Lower bounds for {B}etti
  numbers}, Compositio Math. \textbf{86} (1993), no.~2, 147--158. \MR{1214454
  (94a:13011)}

\bibitem{AvramovBuchweitz00a}
\bysame, \emph{Homological algebra modulo a regular sequence with special
  attention to codimension two}, J. Algebra \textbf{230} (2000), no.~1, 24--67.
  \MR{1774757 (2001g:13032)}

\bibitem{AvramovBuchweitz00}
\bysame, \emph{Support varieties and cohomology over complete intersections},
  Invent. Math. \textbf{142} (2000), no.~2, 285--318. \MR{1794064
  (2001j:13017)}

\bibitem{AvramovGasharovPeeva97}
Luchezar~L. Avramov, Vesselin~N. Gasharov, and Irena~V. Peeva, \emph{Complete
  intersection dimension}, Inst. Hautes \'Etudes Sci. Publ. Math. (1997),
  no.~86, 67--114 (1998). \MR{1608565 (99c:13033)}

\bibitem{AvramovIyengar07}
Luchezar~L. Avramov and Srikanth~B. Iyengar, \emph{Constructing modules with
  prescribed cohomological support}, Illinois J. Math. \textbf{51} (2007),
  no.~1, 1--20. \MR{2346182 (2008j:13036)}

\bibitem{AvramovSun98}
Luchezar~L. Avramov and Li-Chuan Sun, \emph{Cohomology operators defined by a
  deformation}, J. Algebra \textbf{204} (1998), no.~2, 684--710. \MR{1624432
  (2000e:13021)}

\bibitem{Beder14}
Jesse Beder, \emph{The grade conjecture and asymptotic intersection
  multiplicity}, Proc. Amer. Math. Soc. \textbf{142} (2014), no.~12,
  4065--4077. \MR{3266978}

\bibitem{Bergh07}
Petter~Andreas Bergh, \emph{On support varieties for modules over complete
  intersections}, Proc. Amer. Math. Soc. \textbf{135} (2007), no.~12,
  3795--3803 (electronic). \MR{2341929 (2008g:13022)}

\bibitem{Bourbaki83}
Nicolas Bourbaki, \emph{\'{E}l\'ements de math\'ematique}, Masson, Paris, 1983,
  Alg{\`e}bre commutative. Chapitre 8. Dimension. Chapitre 9. Anneaux locaux
  noeth{\'e}riens complets. [Commutative algebra. Chapter 8. Dimension. Chapter
  9. Complete Noetherian local rings]. \MR{722608 (86j:13001)}

\bibitem{CarlsonIyengar12}
Jon Carlson and Srikanth~B. Iyengar, \emph{Thick subcategories of the bounded
  derived category of a finite group},  (2012), arXiv:1201.6536v1.

\bibitem{CelikbasPiepmeyer14}
Olgur Celikbas and Greg Piepmeyer, \emph{Syzygies and tensor product of
  modules}, Math. Z. \textbf{276} (2014), no.~1-2, 457--468. \MR{3150213}

\bibitem{Christensen00}
Lars~Winther Christensen, \emph{Gorenstein dimensions}, Lecture Notes in
  Mathematics, vol. 1747, Springer-Verlag, Berlin, 2000. \MR{1799866
  (2002e:13032)}

\bibitem{DuttaHochsterMcLaughlin85}
Sankar~P. Dutta, M.~Hochster, and J.~E. McLaughlin, \emph{Modules of finite
  projective dimension with negative intersection multiplicities}, Invent.
  Math. \textbf{79} (1985), no.~2, 253--291. \MR{778127 (86h:13023)}

\bibitem{Eisenbud80}
David Eisenbud, \emph{Homological algebra on a complete intersection, with an
  application to group representations}, Trans. Amer. Math. Soc. \textbf{260}
  (1980), no.~1, 35--64. \MR{570778 (82d:13013)}

\bibitem{Gabber95}
Ofer Gabber, \emph{Non-negativity of {S}erre's intersection multiplicities},
  Ex-po{\'s}e {\`a} L'IHES (D{\'e}cember 1995).

\bibitem{Gerko01}
A.~A. Gerko, \emph{On homological dimensions}, Mat. Sb. \textbf{192} (2001),
  no.~8, 79--94. \MR{1862245 (2002h:13024)}

\bibitem{GilletSoule85}
Henri Gillet and Christophe Soul{\'e}, \emph{{$K$}-th\'eorie et nullit\'e des
  multiplicit\'es d'intersection}, C. R. Acad. Sci. Paris S\'er. I Math.
  \textbf{300} (1985), no.~3, 71--74. \MR{777736 (86k:13027)}

\bibitem{Gulliksen74}
Tor~H. Gulliksen, \emph{A change of ring theorem with applications to
  {P}oincar\'e series and intersection multiplicity}, Math. Scand. \textbf{34}
  (1974), 167--183. \MR{0364232 (51 \#487)}

\bibitem{Harris95}
Joe Harris, \emph{Algebraic geometry}, Graduate Texts in Mathematics, vol. 133,
  Springer-Verlag, New York, 1995, A first course, Corrected reprint of the
  1992 original. \MR{1416564 (97e:14001)}

\bibitem{Hartshorne98}
Robin Hartshorne, \emph{Coherent functors}, Adv.in Math. \textbf{140} (1998),
  44--94.

\bibitem{HunekeJorgensen03}
Craig Huneke and David~A. Jorgensen, \emph{Symmetry in the vanishing of {E}xt
  over {G}orenstein rings}, Math. Scand. \textbf{93} (2003), no.~2, 161--184.
  \MR{2009580 (2004k:13039)}

\bibitem{Jorgensen08}
David Jorgensen, \emph{Finite projective dimension and the vanishing of
  $\ext_r(m,m)$}, Comm. Alg. \textbf{36} (2008), no.~12, 4461--4471.

\bibitem{Mehta76}
Vikram~Bhagvandas Mehta, \emph{Endomorphisms of complexes and modules over
  gorenstein rings}, ProQuest LLC, Ann Arbor, MI, 1976, Thesis
  (Ph.D.)--University of California, Berkeley. \MR{2626509}

\bibitem{Miller98}
Claudia Miller, \emph{Complexity of tensor products of modules and a theorem of
  {H}uneke-{W}iegand}, Proc. Amer. Math. Soc. \textbf{126} (1998), no.~1,
  53--60. \MR{1415354 (98c:13022)}

\bibitem{PeskineSzpiro73}
C.~Peskine and L.~Szpiro, \emph{Dimension projective finie et cohomologie
  locale. {A}pplications \`a la d\'emonstration de conjectures de {M}.
  {A}uslander, {H}. {B}ass et {A}. {G}rothendieck}, Inst. Hautes \'Etudes Sci.
  Publ. Math. (1973), no.~42, 47--119. \MR{0374130 (51 \#10330)}

\bibitem{Roberts85}
Paul Roberts, \emph{The vanishing of intersection multiplicities of perfect
  complexes}, Bull. Amer. Math. Soc. (N.S.) \textbf{13} (1985), no.~2,
  127--130. \MR{799793 (87c:13030)}

\bibitem{Stevenson12}
Greg Stevenson, \emph{Subcategories of singularity categories via tensor
  actions},  (2012), arXiv:1105.4698v3.

\end{thebibliography}

\end{document}